\documentclass[hidelinks,12pt]{amsart}
\usepackage[utf8]{inputenc}
\usepackage[left=.9truein,right=.9truein,bottom=.9truein,top=.9truein]{geometry}
\usepackage{mathrsfs,amsfonts, amssymb, amsmath, amscd,enumerate,graphicx,tikz,pgfplots}
\usepackage{stmaryrd,booktabs,mathtools,young,color,MnSymbol,wasysym,mathrsfs,hyperref,tikz-cd}
\usepackage[all]{xy}
\usepackage{xcolor,cancel}

\pgfplotsset{compat=1.14}
\usetikzlibrary{shapes,plotmarks,matrix,arrows}
\hypersetup{pdfborder = {0 0 0}}
\newcommand{\fto}[1]{\stackrel{#1}{\to}}

\newcommand{\isom}{\cong}

\newcommand{\Ext}{\textnormal{Ext}}

\newcommand{\Hom}{\textnormal{Hom}}

\newcommand{\coker}{\textnormal{Coker}\,}
\newcommand{\pic}{\textnormal{Pic}}

\newcommand{\Spec}{\textnormal{Spec}}

\newcommand{\res}{\textnormal{res}}
\newcommand{\rank}{\textnormal{rank}}
\newcommand{\ZZ}{\mathbb Z}
\newcommand{\QQ}{\mathbb Q}
\newcommand{\CC}{\mathbb C}
\newcommand{\PP}{\mathbb P}

\newcommand{\OO}{\mathcal O}

\newcommand{\II}{\mathcal I}
\newtheorem{theo}{Theorem}[section]
\newtheorem{thm}[theo]{Theorem}
\newtheorem*{thm*}{Theorem}
\newtheorem{prop}[theo]{Proposition}
\newtheorem*{prop*}{Proposition}
\newtheorem{rem}[theo]{Remark}
\newtheorem*{rem*}{Remark}
\newtheorem{lemma}[theo]{Lemma}
\newtheorem*{lemma*}{Lemma}
\newtheorem{cor}[theo]{Corollary}
\newtheorem*{cor*}{Corollary}

\newtheorem*{conj*}{Conjecture}

\newtheorem*{ques*}{Question}

\newtheorem*{claim*}{Claim}
\theoremstyle{definition}
\newtheorem{deff}[theo]{Definition}
\newtheorem*{deff*}{Definition}

\newtheorem*{example*}{Example}

\newtheorem*{exmp*}{Example}

\newtheorem*{prob*}{Problem}

\makeatletter
\newtheorem*{rep@theorem}{\rep@title}
\newcommand{\newreptheorem}[2]{%
\newenvironment{rep#1}[1]{%
 \def\rep@title{#2 \ref{##1}}%
 \begin{rep@theorem}}%
 {\end{rep@theorem}}}
\makeatother
\newtheorem{theorem}[theo]{Theorem}
\newtheorem*{theorem*}{Theorem}
\newreptheorem{theorem}{Theorem}

\begin{document}
\title{Irreducibility and Singularities of Some Nested Hilbert Schemes}
\author{Tim Ryan}
\address{Department of Mathematics, University of Michigan, Ann Arbor, MI}
\email{rtimothy@umich.edu}
\author{Gregory Taylor}
\address{Department of Mathematics, Statistics, and Computer Science, University of Illinois at Chicago, Chicago, IL}
\email{gtaylo9@uic.edu}

\begin{abstract}
Let $S$ be a smooth projective surface over $\CC$. We study the local and global geometry of the nested Hilbert scheme of points $S^{[n,n+1,n+2]}$. In particular, we show that $S^{[n,n+1,n+2]}$ is an irreducible local complete intersection with klt singularities. In addition, we compute the Picard group of $S^{[n,n+1,n+2]}$ when $h^1(S,\mathcal{O}_S) = 0$. From the irreducibility of $S^{[n,n+1,n+2]}$, we deduce irreducibility for four other infinite families of nested Hilbert schemes. We give the first explicit example of a reducible nested Hilbert scheme, which allows us to show that $S^{[n_1,\dots,n_k]}$ is reducible for $k > 22$. 
\end{abstract}
\maketitle

\section{Introduction}

Let $S$ be a smooth algebraic surface over $\CC$. For an increasing tuple of positive integers $n_1 < \cdots < n_k$, the \emph{nested Hilbert scheme of points}, denoted $S^{[n_1,\dots,n_k]}$, parametrizes nested subschemes $\xi_{n_1} \subseteq \cdots \subseteq \xi_{n_k}$ of $S$, where $\xi_{n_i}$ is a finite scheme of length $n_i$. In this paper, our primary focus is the local and global geometry of $S^{[n,n+1,n+2]}$, i.e. the scheme parametrizing two-step complete nestings. These results can be understood intuitively as statements about the simultaneous smoothings of such nestings of finite subschemes of $S$.

\begin{thm}\label{thm:TwoStepIntro}
	Let $S$ be a smooth projective surface over $\CC$. 
	\begin{enumerate}
		\item (Theorem \ref{thm:TwoStepIrreducible}) The nested Hilbert scheme $S^{[n,n+1,n+2]}$ is irreducible of dimension $2n+4$.
		\item (Proposition \ref{prop:TwoStepIsLCI}, Corollary \ref{cor:TwoStepRational}) $S^{[n,n+1,n+2]}$ is a local complete intersection with klt singularities.
		\item (Theorem \ref{thm:PicardGroup}) If $h^1(S,\mathcal{O}_S) = 0$, then $\pic(S^{[n,n+1,n+2]}) \cong \pic(S)^3 \oplus \ZZ^3$.
	\end{enumerate}
\end{thm}

To study Hilbert schemes of complete nestings (where $n_i = n_{i-1} + 1$ for all $i$), we systematically employ the \emph{residual point} maps. Let $\xi_n \subseteq \xi_{n+1}$ be subschemes of length $n$ and $n+1$ respectively. The residual point, i.e. the point where the two subschemes disagree, is well-defined, which gives us a map to $S$. As an application of this technique, we use the irreducibility statement in Theorem \ref{thm:TwoStepIntro} to conclude the irreducibility of four additional infinite families of nested Hilbert schemes. In particular, we obtain a more elementary proof that $S^{[n,n+2]}$ is irreducible, which was proved by Geertsen and Hirschowitz using liason methods \cite{GH04}.

\begin{cor}\label{cor:IrreducibilityIntro}
	Let $S$ be a smooth projective surface over $\CC$. The following nested Hilbert schemes are irreducible 
	\begin{enumerate}
	    \item (Corollary \ref{cor:[n,n+2] irreducible}) $S^{[n,n+2]}$,
	    \item (Proposition \ref{prop: [1,n,n+1,n+2] irred}) $S^{[1,n,n+1,n+2]}$,
	    \item (Corollary \ref{cor: [1,n+1,n+2],[1,n,n+2] irred}) $S^{[1,n+1,n+2]}$, and $S^{[1,n,n+2]}$.
	\end{enumerate}
\end{cor}

We prove that $S^{[n,n+1,n+2]}$ has klt singularities by constructing an explicit resolution of singularities. To state our result, we define
\[
	Z_{[n,n+1]} := \{(s,\xi_{n},\xi_{n+1}) \in S \times S^{[n,n+1]} : s \in \xi_{n+1}\}.
\]
It is not hard to show that $Z_{[n,n+1]}$ has two irreducible components (see Proposition \ref{prop:Zn,n+1,n+2IsReducible}). It follows from general principles that $S^{[n,n+1,n+2]}$ is the blowup of $S \times S^{[n,n+1]}$ at this reducible subscheme. We construct our resolution by blowing up each component separately.

\begin{thm}\label{thm:ResolutionIntro}
	Let $Z_{[n,n+1]} = W_1 \cup W_2$ be the irreducible components of $Z_{[n,n+1]}$ with $W_1 \cong S^{[1,n,n+1]}$ and $W_2 \cong S^{[n,n+1]}$. Let $X_1$ be the blowup of $S \times S^{[n,n+1]}$ at $W_1$, and $X_2$ be the blowup of $X_1$ at the proper transform $\overline{W}_2$ of $W_2$. Then $X_1,X_2$ are smooth, and $X_2$ is a resolution of singularities of $S^{[n,n+1,n+2]}$.
\end{thm}

Our construction of the resolution in Theorem \ref{thm:ResolutionIntro} is heavily influenced by Tikhomirov's proof that $S^{[n,n+1]}$ is smooth \cite{Tik}. As such, we provide an exposition of his proof in Section \ref{sec: smoothness of one step}. We included this section in part because the published version \cite{Tik} has not been cited in the literature. Although the smoothness of $S^{[n,n+1]}$ is typically attributed to Tikhomirov, the result is cited as a preprint under a different title in every paper known to the authors. There is another proof of this result due to Cheah, where he reduces to the case $S = \PP^2$ and proceeds by checking smoothness at the fixed points of the $\CC^*$-action \cite{CheahThesis}.

This paper builds on a great deal of work concerning nested Hilbert schemes of points on surfaces. It is classical that the Hilbert scheme of points $S^{[n]}$ is a smooth projective variety of dimension $2n$ \cite{Fogarty}. These Hilbert schemes exhibit rich behavior and often arise in connections with other fields \cite{ABCH, Beauville, Haiman}. The first non-trivial nested Hilbert scheme is $S^{[n,n+1]}$, which can be used to study $S^{[n]}$ inductively \cite{FGAexp}. In fact, $S^{[n,n+1]}$ is also smooth and irreducible \cite{CheahThesis}. Ellingsrud-Str{\o}mme \cite{ES98} used $S^{[n,n+1]}$ in an intersection theoretic calculation on the Hilbert scheme which completed Nakajima's computation of the representation of Heisenberg algebra using the cohomology of the Hilbert scheme \cite{Nakajima}, and Ryan-Yang applied $S^{[n,n+1]}$ to syzygies of algebraic surfaces \cite{RY}. Another important nested Hilbert scheme is $S^{[1,n]}$ which is the universal family over $S^{[n]}$. The universal family $S^{[1,n]}$ for $S$ a K3 surface played a crucial role Voisin's papers on Green's conjecture \cite{Voisin1, Voisin2}. The universal family $S^{[1,n]}$ is irreducible and has at worst rational singularities \cite{Song16}.

Our knowledge of more general nested Hilbert schemes is sparse. For example, we do not have a characterization of the tuples $(n_1,\dots,n_k)$ for which $S^{[n_1,\dots,n_k]}$ is irreducible. It has long been known that $S^{[n]}$, $S^{[1,n]}$, and $S^{[n,n+1]}$ are irreducible \cite{Fogarty, ES98}.
More recently, results of Goddard and Goodwin on spaces of commuting varieties of parabolic subalgebras imply that $S^{[n,n+k]}$ is irreducible for $k\leq 6$ and that $S^{[n_1,\dots,n_k]}$ is irreducible if $n_k \leq 16$ \cite{GG18}. 
See work of Bulois and Evain \cite{BuloisEvain} for a more through exposition of the general connection between nested Hilbert schemes and parabolic subalgebras. 
Similar methods were used in Boos and Bulois \cite{BB} to show that the punctual nested Hilbert schemes had larger than expected dimensional components in the case $\mathbb{A}_0^{2[n,n+k]}$ where $k>5$.
The methods employed in these recent results is Lie theoretic, while our approach is more geometric. There is undoubtedly much more to be discovered through the interaction of these approaches.

Given that general Hilbert schemes can be quite pathological \cite{Vakil}, it is expected that the nice behavior exhibited in Theorem \ref{thm:TwoStepIntro} fails as the number of steps in the nesting grows. We confirm this expectation by providing an explicit example of a reducible nested Hilbert scheme. This construction allows us to give a sufficient condition on the numbers $n_1,\dots,n_k$ for the reducibility of $S^{[n_1,\dots,n_k]}$.

\begin{thm}\label{thm:ReducibleExample}
    The nested Hilbert scheme $S^{[n_1,\dots,n_k]}$ is reducible for $k > 22$.
\end{thm}

While this theorem does not completely settle the question of irreducibility for nested Hilbert schemes, it does significantly narrow the search for the range of tuples $(n_1,\dots,n_k)$ for which the nested Hilbert scheme $S^{[n_1,\dots,n_k]}$ is irreducible.

The paper is organized as follows. Section \ref{sec: background} provides the necessary background on nested Hilbert schemes of points. In Section \ref{sec: irreducibility}, we prove our results on (ir)reducibility of nested Hilbert schemes.
Section \ref{sec: smoothness of one step} presents a simplified version of Timhomirov's proof of the smoothness of $S^{[n,n+1]}$, as the strategy motivates our proof of the resolution of $S^{[n,n+1,n+2]}$. 
Section \ref{sec:ResolutionAndSings} contains all of our results on singularities. 
Finally, in Section \ref{sec: geo results}, we compute the Picard group and the class of the canonical divisor of $S^{[n,n+1,n+2]}$.

\subsection*{Acknowledgements} 
The authors would like to thank Kevin Tucker and Izzet Coskun for valuable discussions during the preparation of this article. The second author was supported by NSF grant 1246844.

\section{Background} \label{sec: background}
In this section, we review the necessary background on nested Hilbert schemes and singularities.
For an introductions to nested Hilbert schemes, see \cite{CheahThesis}, and for the relevant singularities, see \cite{KollarBook}.
Throughout this section, let $S$ be a smooth irreducible projective surface over $\mathbb{C}$ with $h^1(S,\mathcal{O}_S) = 0$.

\subsection{Nested Hilbert Schemes of Points} The parameter space of unordered $n$-tuples of points on $S$ is the symmetric product $S^{(n)}$, which is the cartesian product $S^n$ mod the action of the symmetric group $\mathfrak{S}_n$.
Unfortunately, the symmetric product is singular along the locus of double (and more) points.
The Hilbert scheme of points, $S^{[n]}$ which parameterizes subschemes of $S$ with constant Hilbert polynomial $n$, is a resolution of singularities of the symmetric product via the Hilbert-Chow morphism \[\mathrm{hc}:S^{[n]} \to S^{(n)} ,\]
which maps a subscheme to its support with multiplicities.

The nested Hilbert scheme $S^{[n_1,\dots,n_k]}$ parametrizes $k$-tuples of subschemes with respective constant Hilbert polynomials $n_1$ through $n_k$ which are nested inside of each other, i.e. 
\[S^{[n_1,\dots,n_k]} = \{(\xi_{n_1},\dots,\xi_{n_k}) \subset S^{[n_1]} \times \dots \times S^{[n_k]}: \xi_{n_1} \subset \dots \subset \xi_{n_k}\}.\]
In other words, the nested Hilbert scheme is the incidence correspondence inside of the product $S^{[n_1]}\times \dots \times S^{[n_k]}$, which defines the scheme structure on it.
As a subscheme of the product, $S^{[n_1,\dots,n_k]}$ comes equipped with the projections 
\[\pi_{n_i} = \mathrm{pr}_{n_i}\vert_{S^{[n_1,\dots,n_k]}}: S^{[n_1,\dots,n_k]} \to S^{[n_i]} \text{ and}\]
\[\phi_{n_i}: S^{[n_1,\dots,n_k]} \to S^{[n_1,\dots,\widehat{n_i},\dots n_k]}.\]
That is $\pi_{n_i}(\xi_{n_1},\dots,\xi_{n_k}) = \xi_{n_i}$, and $\phi_{n_i}(\xi_{n_1},\dots,\xi_{n_k}) = (\xi_{n_1},\dots,\widehat{\xi_{n_i}},\dots,\xi_{n_k})$.
They also come with residual maps
\[\mathrm{res}_{[n_i,n_j]}:S^{[n_1,\dots,n_k]} \to S^{(n_j-n_i)}\]
which send a nest of subschemes to the support (with multiplicity) of the difference of two subschemes in the nesting.
This morphism is particularly useful in the case that $n_j-n_i=1$, in which case \[S^{(n_j-n_i)} = S^{(1)} = S^{[1]} = S.\]
In the case when the nested Hilbert scheme has only two indices, we drop the subscript on the map. 
Note, the notation for each of these maps is slightly abusive as they do specify which nested Hilbert scheme is the domain. However, the domain will be clear when we use them so no confusion arises.

In order to study the nested Hilbert schemes, it is often useful to consider the universal families over them.
These are the (possibly reducible) incidence correspondence \[Z_{[n_1,\dots,n_k]} = \{(p,(\xi_{n_1},\dots,\xi_{n_k})): p \in \mathrm{supp}(\xi_{n_k}) \}.\]

\begin{prop}
    Suppose $\{m_1,\dots,m_s\} \subseteq \{n_1,\dots, n_t\} \subseteq \mathbb{N}$, then the natural map $S^{[n_1,\dots,n_t]} \to S^{[m_1,\dots,m_s]}$ is surjective.
\end{prop}
\begin{proof}
    It suffices to show that for any $(\xi_k,\xi_{k+\ell}) \in S^{[k,k+\ell]}$, there is a chain $\xi_k \subseteq \xi_{k+1} \subseteq \cdots \subseteq \xi_{k+\ell}$ where $\xi_{i} \in S^{[i]}$. This is precisely the statement of \cite[Proposition 3.14]{BuloisEvain}.
\end{proof}

\subsection{The Nested Hilbert Scheme $S^{[n,n+1,n+2]}$}

We rely heavily on the description of the map $(\phi_{n+2}, \res_{[n+1,n+2]}) : S^{[n,n+1,n+2]} \to S^{[n,n+1]} \times S$ as the projectivization of an ideal sheaf. To see this, we first describe the map $(\phi_{n+1}, \res) : S^{[n,n+1]} \to S^{[n]} \times S$ for which the notation is less overbearing. 

Note that elements in $S^{[n,n+1]}$ correspond to elementary transformations \cite{ES98}. Given a pair of subschemes $(\xi_n,\xi_{n+1}) \in S^{[n,n+1]}$, we have an exact sequence
\[
    0 \to \II_{\xi_{n+1}} \to \II_{\xi_n} \to \CC(p) \to 0
\]
where $p = \res(\xi_n,\xi_{n+1})$. Such sequences are called elementary transformations \cite{EL99}. Conversely for any $p \in S$, the kernel of a surjection $\II_{\xi_n} \to \CC(p)$ is an ideal $\II_{\xi_{n+1}}$ of colength $n+1$ in $\OO_S$ with $\res(\xi_n,\xi_{n+1}) = p$. Thus, the space of pairs $(\xi_n,\xi_{n+1}) \in S^{[n,n+1]}$ with $\res(\xi_n,\xi_{n+1}) = p$ is isomorphic to $\PP(\II_{\xi_n} \otimes_{\OO_S} \CC(p))$. 

This construction globalizes to show that $S^{[n,n+1]} = \PP(\II_{Z_{[n]}})$ \cite{ES98}. We give a sketch of why this is true. Recall from \cite{CheahThesis} that $S^{[n,n+1]}$ represents the functor of nested flat families, i.e.
\[
    \Hom(X,S^{[n,n+1]}) =   \left\{ \begin{array}{c} \text{Flat families } \Xi_n \subseteq \Xi_{n+1} \subseteq X \times S \\ \text{such that the fibers of } \Xi_n,\Xi_{n+1} \\
    \text{over } X \text{ have length } n,n+1 \text{ respectively}\end{array} \right\}.
\]
Let us see why a map $f : X \to \PP(\II_{Z_n})$ gives such a nesting. By the universal property of the Hilbert scheme, the composition of $f$ with with projection map gives us the following diagram
\begin{center}
    \begin{tikzcd} 
        \Xi_n \arrow[dr, "\text{flat}"'] \arrow[r, hook] & X \times S \arrow[d, "\text{pr}"'] \\
        & X \arrow[u, bend right, "\sigma"']
    \end{tikzcd}
\end{center}
Then $f$ gives us a surjection from $\II_{\Xi_n} \to \OO_{\sigma(X)} \to 0$ (by pulling back a similar surjection from $X \times S^{[n]} \times S$ involving the ideal sheaf $\II_{Z_{[n]}}$). The kernel of this map of sheaves is an ideal sheaf $\II_{\Xi_{n+1}}$ which cuts out the desired flat family of length $n+1$ subschemes of $S$.

Using essentially identical formalism, one can show that the map $S^{[n,n+1,n+2]} \to S^{[n,n+1]} \times S$ is the projectivization of the ideal sheaf of $\II_{Z_{[n,n+1]}}$. Let us review an irreducibility criteria for such projectivizations. A good reference for this material is \cite[Section 3]{ES98}. Let $X$ be an irreducible variety with a subscheme $W \subseteq X$ of codimension 2. Further, assume that the ideal sheaf $\II_W$ of $W$ in $X$ admits a resolution
\[
    0 \to \mathcal{A} \to \mathcal{B} \to \II_W \to 0
\]
where $\mathcal{A}$ and $\mathcal{B}$ are locally free. For example, such a resolution exists if $X$ is smooth and $W$ is Cohen-Macaulay. Note that this resolution of length two implies that $\II_W \cdot \PP(\II_W)$ is locally principal, so there is a natural map $\PP(\II_W) \to \mathrm{Bl}_WX$. We define loci
\[
    W_i := \{w\in W : \II_W \text{ is minimally generated by } i \text{ elements at }w \}.
\]
Consider the projection map $\PP(\II_W) \to X$. If $w \in W_i$, then the fiber of the projection map over $w$ is isomorphic to $\PP^{i-1}$.

\begin{prop}[{\cite[Proposition 3.2]{ES98}}]\label{prop:generalIrreducibilityCriterion}
    Let $W,X$ be as above. 
    \begin{enumerate}
        \item If $W_i$ has codimension at least $i \geq 2$ in $X$ for each $i$, then $\PP(\II_W)$ is irreducible and $\PP(\II_W) \cong \mathrm{Bl}_WZ$.
        \item Assume $W$ is irreducible. If $W_i$ has codimension at least $i+1$ in $X$ for each $i\geq 3$, then the exceptional divisor in $\PP(\II_W) = \mathrm{Bl}_WX$ is irreducible.
    \end{enumerate}
\end{prop}

In the situation where $\PP(\II_W)$ is irreducible, this construction also appears in \cite[Section 3.1.2]{JiangLeung} where they give an orthogonal decomposition of the derived category of the blowup. Belmans-Krug applied this result to produce an orthogonal decomposition of the derived category of $S^{[n,n+1]}$ \cite{BelmansKrug}.

\subsection{Singularities}

Let $X$ be a variety with resolution of singularities $f : Y \to X$. We say that $X$ has \emph{rational singularities} if $f_*\OO_Y = \OO_X$ and $R^if_*\OO_X = 0$ for $i > 0$. Furthermore, $X$ is rational if:
\begin{enumerate}
    \item $X$ is Cohen-Macaulay and
    \item $f_*\omega_Y = \omega_X$, where $\omega_Y,\omega_X$ are the canonical modules on $Y,X$ respectively.
\end{enumerate}
This is proved in \cite[Lemma 1]{Kovacs}.

Now suppose $X$ be a normal variety with canonical divisor $K_X$. Furthermore, assume that $K_X$ is $\QQ$-Cartier (i.e. that $X$ is $\QQ$-Gorenstein\footnote{Note that a $\QQ$-Gorenstein variety need not be Cohen-Macaulay.}). Let $f : Y \to X$ be a log resolution. We can write $K_Y \equiv f^*K_X + \sum a(E;X)E$ where the sum is taken over prime divisors $E$ over $X$. Then $X$ is \emph{log terminal} if $a(E;X) > -1$ for each $E$. Log terminal singularities are necessarily rational \cite{Elkik, Kovacs} and thus Cohen-Macaulay. If $X$ is Gorenstein with rational singularities, then $X$ is necessarily log terminal.

\section{Irreducibility Results} \label{sec: irreducibility}
In this section, we  prove that the nested Hilbert schemes $S^{[n,n+1,n+2]}$ and $S^{[1,n,n+1,n+2]}$ are irreducible. As corollaries, we show that the nested Hilbert schemes $S^{[1,n,n+1]}$ and $S^{[1,n,n+2]}$ are irreducible as well as providing a new proof that the nested Hilbert scheme $S^{[n,n+2]}$ is irreducible.
Finally, we provide a limit on how far this type of result can be proven by giving the first example of a reducible nested Hilbert scheme of points on a surface. 

To prove that $S^{[n,n+1,n+2]}$ is irreducible, we want to apply Proposition 3.2 of \cite{ES98}. 
Since $S^{[n,n+1,n+2]}$ is the projectivization of the ideal sheaf of the universal family 
\[ Z_{[n,n+1]} := \{(p,\xi_n,\xi_{n+1}) \in S\times S^{[n,n+1]} : p \in \mathrm{Supp}(\xi_{n+1})\},\] 
we need to prove two things to satisfy the hypotheses of that result.
First, we need to show that the universal family has codimension two in $S \times S^{[n,n+1]}$ and $\mathcal{O}_{Z_{[n,n+1]}}$ has local projective dimension two over $\mathcal{O}_{S \times S^{[n,n+1]}}$.
Second, we need to bound the dimensions of the loci in $S \times S^{[n,n+1]}$ where the ideal sheaf of the universal family needs $i$ generators.

The first of these necessary conditions is relatively straightforward to show.

\begin{lemma}\label{lem: codim two}
$Z_{[n,n+1]}$ has codimension two in $S \times S^{[n,n+1]}$, and $\mathcal{O}_{Z_{[n,n+1]}}$ has local projective dimension two over $\mathcal{O}_{S\times S^{[n,n+1]}}$.
\end{lemma}

\begin{proof}
Since the projection map $Z_{[n,n+1]} \to S^{[n,n+1]}$ is finite, the dimension of $Z_{[n,n+1]}$ is $2n+2$, i.e. it is codimension two in $S \times S^{[n,n+1]}$.

The projection map is also flat.
Since $Z_{[n,n+1]}$ admits a finite, flat map to a smooth variety, $Z_{[n,n+1]}$ is Cohen-Macaulay \cite[Corollary to Theorem 23.3]{Matsumura}.
Since $Z_{[n,n+1]}$ is a codimension two Cohen-Macaulay subscheme of the smooth variety $S^{[n,n+1]} \times S$, the local ring $\OO_{Z_{[n,n+1]}, \zeta}$ is a Cohen-Macaulay module over the regular ring $R = \OO_{S \times S^{[n,n+1]},\zeta}$ for any $\zeta = (p,\xi_n,\xi_{n+1}) \in Z_{[n,n+1]}$. 
As a result, $\OO_{Z_{[n,n+1]}, \zeta}$ has a minimal free resolution
    \begin{center}
        \begin{tikzcd}
            0 \arrow[r] & \bigoplus^{b_2} \arrow[r, "M"] R & \bigoplus^{b_1} R \arrow[r] & R \arrow[r] & \OO_{Z_{[n,n+1]}, \zeta} \arrow[r] & 0
        \end{tikzcd}
    \end{center}
whose length is two by the Auslander-Buchsbaum formula, i.e. $\mathcal{O}_{Z_{[n,n+1]}}$ has local projective dimension two over $\mathcal{O}_{S \times S^{[n,n+1]}}$.
\end{proof}

The second condition is a bit more complicated to show. 
Let us more explicitly state what we need to show.
Denote the locus where the ideal of $Z_{[n,n+1]}$ needs $i$ generators by \[W_{i,[n,n+1]} = \{(p,Z',Z) \in S\times S^{[n,n+1]}: I_{Z_{[n,n+1]}} \text{ needs }i\text{ generators at } (p,Z',Z)\}.\]
We note that this is precisely the number of generators of the ideal of $Z$ at the point $p$.
We need to show that $\mathrm{codim}(W_{i,[n,n+1]}) \geq i$ for all $i\geq 2$.

\subsection*{Dimension bounds for loci in the universal families}
Since the dimension bounds we prove in this section are local, we restrict to the case of $S = \mathbb{A}^2$.
Using the forgetful map $(\mathrm{id},\pi_n):S\times S^{[n,n+1]} \to S \times S^{[n]}$, we want to bound the dimensions of $W_{i,[n]}$ starting from dimensions of the analogous loci in $S \times S^{[n]}$, i.e. \[W_{i,[n]} = \{(p,Z) \in S\times S^{[n]}: I_{Z_{[n]}} \text{ needs exactly }i\text{ generators at } (p,Z)\}.\]
In order to do that, we have to bound both the dimension of $W_{i,[n,n+1]}$ and the dimensions of the fibers over it.
To get a sharp enough bound for the dimension of $W_{i,[n,n+1]}$, we need to bound the dimensions of the $W_{i,[n]}$ more sharply than in \cite{ES98}.
We first show when these are empty. $W_{1,[n]}$ is always nonempty and of dimension $2n+2$ as it is the complement of the universal family.
For higher $i$, we get the following lemma.
\begin{lemma}
\label{lem: 1 step emptiness}
$W_{i,[n]} = \varnothing$ for $i\geq 2$ if  $n< \binom{i}{2}$.  
\end{lemma}

\begin{proof}
Consider a pair $(p,\xi) \in S \times S^{[n]}$.
Denote by $l$ the length of $\xi$ supported at $p$.
If $\binom{i-1}{2} \leq l <\binom{i}{2}$, then $\xi$ locally lies on a degree $i-2$ curve at $p$.
By Corollary 3.9 in \cite{E05},  $\xi$ is locally cut out by at most $i-1$ equations. 
Thus, $(p,\xi)$ is in $W_{j,[n]}$ for some $j\leq i-1$.
It follows immediately that if $n< \binom{i}{2}$, then $W_{i,[n]} = \varnothing$ as desired.
\end{proof}

We want bound the dimension in the case these are nonempty.

\begin{lemma}\label{lemma:betterWiBound}
The codimension of $W_{i,[n]}$ is at least $\binom{i}{2}+1$ for $i\geq 2$ if $n\geq \binom{i}{2}$.
\end{lemma}

\begin{proof}
Consider a pair $(p,\xi) \in S \times S^{[n]}$.
Let $\eta$ be the subscheme of $\xi$ which is supported at $p$, and denote the length of $\eta$ by $l$.
The dimension of length $l$ schemes supported at a point is $l-1$ \cite{B77}.
Since the residual scheme of $\xi$ to $\eta$ has length $n-l$, it is a point of $\left(S\slash \{p\}\right)^{[n-l]}$, which has dimension $2(n-l)$.
Thus, the set of schemes of length $n$ with length $l$ supported at $p$ is dimension $2n-l-1$.
By varying the point $p$, we see that the dimension of schemes with exactly $l$ points supported at a single (moving) point and $(n-l)$ points supported elsewhere is dimension $2n-l+1$.
This dimension is maximized by minimizing $l$, but we again note that to be in $W_{i,[n]}$, we must have that $l \geq \binom{i}{2}$.
Thus, the dimension of  $W_{i,[n]}$ is at most $2n-\binom{i}{2}+1$ if $n\geq \binom{i}{2}$, and the result follows.
\end{proof}

Now that we have bounded the dimension of $W_{i,[n]}$, we need to bound the dimensions of the preimage of $W_{i,[n]}$ in $S \times S^{[n,n+1]}$ the forgetful map.
However, the dimension of the fibers does not have a consistent bound over all of $W_{i,[n]}$.
In order to work around this, we need a new collection of sets, where the fiber dimension has a consistent bound, which are 
\[W'_{i,[n]} = \{(p,\xi)\in S \times S^{[n]}:i= \max\{j: (q,\xi)\in W_{j,[n]} \text{ for some } q\}\}.\]
In words, the locus $W_{i,[n]}'$ contains pairs $(p,\xi)$ such that $\II_\xi$ requires $i$ generators at some point of its support (not necessarily at $p$) and no more than $i$ generators at any other point.
\begin{lemma}\label{lem: constant fiber dimension}
The dimension of the fibers over $W'_{i,[n]}$ in the forgetful map $(\mathrm{id},\pi_n): S\times S^{[n,n+1]} \to S \times S^{[n]}$ are dimension $2$ if $i$ is $1$ or $2$ and dimension $i-1$ if $i\geq 3$.
\end{lemma}

\begin{proof}
Recall that the fibers of the map $\phi: S^{[n,n+1]} \to S^{[n]}\times S$ are $(i-1)$ over any $\xi$ such that $(p,\xi) \in W'_{i,[n]}$ for any $p$.
Then the fiber of $f$ over a point $(p,\xi) \in W'_{i,[n]}$ is the union over all fibers of $\phi$ over points $(q,\xi)$ where $q\in S$.
All but finitely many of these fibers are single points.
The union of those fibers is a 2-dimensional variety so the entire fiber is the union of this surface and finitely many projective spaces, which are all dimension at most $i-1$.
Thus, the entire fiber is dimension at most $i-1$ for $i\geq 3$ and is dimension $2$ if $i=1$ or $i=2$.
\end{proof}

We want to consider the preimage of $W_{i,[n]}$ as the union over the preimages of the sets $W_{i,[n]} \cap W'_{j,[n]}$ where $j\geq i$.
Since we know the dimension of the fibers of $W_{i,[n]}'$, we need to compute the dimension of the intersections $W_{i,[n]} \cap W'_{j,[n]}$.
We again first note when this intersection is empty.

\begin{lemma}\label{lem: intersection emptiness}
$W_{i,[n]} \cap W'_{j,[n]} = \varnothing$ if $j<i$
\end{lemma}
This lemma is immediate since a point $(p,Z) \in W_{i,[n]}$ requires $i$ generators at $p$.

\begin{lemma}\label{lem: intersection dim}
The codimension of $W_{i,[n]} \cap W'_{j,[n]}$ is at least $\binom{j}{2}+\binom{i}{2}$ for $j>i\geq 2$ and at least $\binom{i}{2}+1$ for $j=i\geq 2$.
\end{lemma}
\begin{proof}
Since $W_{i,[n]} \cap W'_{i,[n]}$ is open in $W_{i,[n]}$ so has the same dimension bound.
Let $(p,\xi) \in W_{i,[n]} \cap W'_{j,[n]}$ with $j>i\geq 2$.
Then $\xi$ requires $i$ local generators at $p$ and $j$ local generators at some other point $q$.
Let $\ell$ and $m$ be the length of $\xi$ supported at $p$ and $q$, respectively. 
As before, we can see the dimension of the locus of schemes with those lengths at those points is $2n+2-\ell-m$.
This is maximized by minimizing $\ell$ and $m$, but those are each at least $\binom{i}{2}$ and $\binom{j}{2}$, respectively.
This means that the dimension of $W_{i,[n]} \cap W'_{j,[n]}$ is bounded above by $2n+2-\binom{j}{2}-\binom{i}{2}$ for $j>i\geq 2$, and the result follows.
\end{proof}

We can combine Lemmas \ref{lem: constant fiber dimension} and \ref{lem: intersection dim} to give the dimensions of the fibers over $W_{i,[n]}$.

\begin{lemma}
The codimension of the preimage of the forgetful map over $W_{i,[n]}$ is at least $2$ if $i=2$ and at least $\binom{i-1}{2}+3$ if $i\geq 3$. 
\end{lemma}

\begin{proof}
Since $W_{i,[n]} = \bigcup_{j\geq i} W_{i,[n]} \cap W'_{j,[n]}$ and there are finitely many nonempty intersection, it suffices to show that the codimension of the fiber over $W_{i,[n]} \cap W'_{j,[n]}$ satisfies the desired bounds.

The codimension of $W_{2,[n]} \cap W'_{2,[n]}$ is $2$ and the codimension of $W_{2,[n]} \cap W'_{j,[n]}$ is at least $\binom{j}{2}+1$ for $j>2$ by Lemma \ref{lem: intersection dim}.
Since the fibers over $W'_{2,[n]}$ are $2$-dimensional and the fibers over $W'_{j,[n]}$ are $(j-1)$-dimensional for $j\geq 3$ by Lemma \ref{lem: constant fiber dimension}, the codimension of the fiber over $W_{2,[n]}$ is at least the minimum of $2$ and $\binom{j-1}{2}+1$ with $j>2$, which is $2$.

Similarly, if $i\geq 3$, the codimension of $W_{i,[n]} \cap W'_{i,[n]}$ is $\binom{i}{2}+1$ and the codimension of $W_{i,[n]} \cap W'_{j,[n]}$ is $\binom{j}{2}+\binom{i}{2}$ for $j>i$ by Lemma \ref{lem: intersection dim}.
Since the fibers over $W'_{i,[n]}$ are $(i-1)$-dimensional and the fibers over $W'_{j,[n]}$ are $(j-1)$-dimensional for $j>i$ by Lemma \ref{lem: constant fiber dimension}, the codimension of the fiber over $W_{i,[n]}$ is the minimum of $\binom{i-1}{2}+3$ and $\binom{i}{2}+\binom{j-1}{2}+2$ with $j>i$, which is $\binom{i-1}{2}+3$.
\end{proof}

We finally can consider a dimension bound for $W_{i,[n,n+1]}$.
We again first note when these are empty.
\begin{lemma}
\label{lem: 1 step dim}
$W_{i,[n,n+1]} = \varnothing$ for $i\geq 1$ if  $n+1< \binom{i}{2}$.  
\end{lemma}

The proof is entirely analogous to the proof of Lemma \ref{lem: 1 step emptiness}.
We then again bound the dimension in the case these are nonempty.

\begin{lemma}\label{lem: 2 step dim}
The codimensions of $W_{1,[n,n+1]}$ and $W_{2,[n,n+1]}$ are $0$ and $2$ respectively.
The codimension of $W_{i,[n,n+1]}$ is at least $\binom{i-2}{2}+3$ for $i \geq 3$.
\end{lemma}

\begin{proof}
We first consider the base cases of $i=1$ and $i=2$.
$W_{1,[n,n+1]}$ is an open set of the entire product so it has codimension $0$.
Similarly, $W_{2,[n,n+1]}$ is an open set of the universal family so it has codimension $2$.

For the general case, we recall the forgetful map, $(\mathrm{id},\pi_n): S\times S^{[n,n+1]}\to S\times S^{[n]}$, which forgets the length $n+1$ scheme, i.e. $(p,\xi_n,\xi_{n+1}) \mapsto (p,\xi_n)$.
By \cite{ES98}, we know that $W_{i,[n,n+1]}$ is contained in the preimage of $W_{i-1,[n]}\cup W_{i,[n]}\cup W_{i+1,[n]}$
The minimum codimension of those three fibers is $\binom{i-2}{2}+3$ for $i\geq 4$.

The analoguous bound for the codimension of $W_{3,[n,n+1]}$ is at least 2.
However, an open set of the preimage of $W_{2,[n,n+1]}$ is an open subset of the universal family which is not in $W_{3,[n,n+1]}$, so $W_{3,[n,n+1]}$ has codimension at least 3.
\end{proof}

Note that the previous lemma shows that the codimension of $W_{i,[n,n+1]}$ is at least $i$ for all $i$.

\subsubsection{Irreducibility proofs}

We have seen that $S^{[n,n+1,n+2]}$ is $\PP(\II_{Z_{[n,n+1]}})$, and below, we use the dimension estimates of the previous subsection to prove that $S^{[n,n+1,n+2]}$ is irreducible. The subscheme $Z_{[n,n+1]}$ has two irreducible components (see Proposition \ref{prop:Zn,n+1IsReducible}). One of its components is 
\[
    Z_{[n,n+1]}^n =  \{(p,\xi_n,\xi_{n+1}) \in S \times S^{[n,n+1]} : p \in \xi_n \subseteq \xi_{n+1}\} \cong S^{[1,n,n+1]}. 
\]
Our resolution of singularities for $S^{[n,n+1,n+2]}$ involves $\PP(\II_{Z_{[n,n+1]}^n})$. 

\begin{thm}\label{thm:TwoStepIrreducible} 
  Let $S$ be a smooth surface.
  \begin{enumerate}
     \item $S^{[n,n+1,n+2]}$ is irreducible of dimension $2n+4$ and is isomorphic to the blowup of $Z_{[n,n+1]}$ in $S\times S^{[n,n+1]}$.
     \item $\PP(\II_{Z_{[n,n+1]}^n})$ is irreducible of dimension $2n+4$ and is isomorphic to the blowup of $Z_{[n,n+1]}^n$ in $S \times S^{[n,n+1]}$. Moreover, the exceptional divisor is irreducible.
  \end{enumerate}
\end{thm}

\begin{proof}
    For both parts, we use Proposition \ref{prop:generalIrreducibilityCriterion}. First, we prove part (a). By Lemma \ref{lem: codim two}, $Z_{[n+1,n+2]}$ is codimension two in $S \times S^{[n,n+1]}$ and $\mathcal{O}_{Z_{[n,n+1]}}$ has local projective dimension two over $\mathcal{O}_{S \times S^{[n,n+1]}}$. By Lemma \ref{lem: 2 step dim}, $W_{i,[n,n+1]}$ is codimension at least $i$ in $S\times S^{[n,n+1]}$, so we can immediately apply Lemma 3.2 of \cite{ES98} to conclude the first statement. 

    For part (b), let $W_{i,[n,n+1]}^n \subseteq S \times S^{[n,n+1]}$ be the locus of triples $(p,\xi_n,\xi_{n+1})$ such that $p \in \xi_n$ and $\II_{\xi_n}$ requires $i$ generators at $p$. It suffices to show that this locus has codimension at least $i+1$. Consider the map $f = \mathrm{id} \times \phi_{n+1} : S \times S^{[n,n+1]} \to S \times S^{[n]}$. Note that $W_{i,[n,n+1]}^n = f^{-1}(W_{i,[n]})$. By Lemma \ref{lemma:betterWiBound}, we have $\dim(W_{i,[n]}) \leq 2n+1 - {i \choose 2}$. The fiber of $f$ over a point in $W_{i,[n]}$ has dimension $i-1$. Thus,
    \begin{align*}
        \dim(W_{i,[n,n+1]}^n) \leq 2n - {i \choose 2} + i.
    \end{align*}
    Equivalently,
    \begin{align*}
        \mathrm{codim}_{S\times S^{[n,n+1]}}(W_{i,[n,n+1]}^n) \leq {i \choose 2} - i + 4
    \end{align*}
    which is at least $i+1$ for $i \geq 3$.
\end{proof}

This result has an easily corollary which recovers a result of \cite{GH04}, which was generalized by \cite{GG18}.

\begin{cor}\label{cor:[n,n+2] irreducible}\cite{GH04}
$S^{[n,n+2]}$ is irreducible of dimension $2n+4$.
\end{cor}

\begin{proof}
The forgetful map $\phi_{n+1}: S^{[n,n+1,n+2]} \to S^{[n,n+2]}$ is surjective and generically finite so the result is immediate from the theorem.
\end{proof}

In order to prove the irreducibility of $S^{[1,n,n+1,n+2]}$, we need a lemma describing the universal family.

\begin{prop}\label{prop:Zn,n+1,n+2IsReducible}
	The subscheme $Z_{[n,n+1,n+2]}$ is covered by three loci, two of which are isomorphic to $S^{[n,n+1,n+2]}$ and one of which is isomorphic to $S^{[1,n,n+1,n+2]}$.
\end{prop}
\begin{proof}
	There is an obvious map $S^{[1,n,n+1,n+2]} \to Z_{[n,n+1,n+2]}$ whose image is the collection of tuples $(p,\xi_n,\xi_{n+1},\xi_{n+2})$ with $p \in \xi_n$. This map is an isomorphism onto its image, which we denote $Z^n_{[n,n+1,n+2]}$. 
	There are other natural maps $S^{[n,n+1,n+2]} \to Z_{[n,n+1,n+2]}$ given by $(\xi_n,\xi_{n+1},\xi_{n+1}) \mapsto (\res(\xi_n,\xi_{n+1}),\xi_n,\xi_{n+1},\xi_{n+2})$ and by $(\xi_n,\xi_{n+1},\xi_{n+1}) \mapsto (\res(\xi_{n+1},\xi_{n+2}),\xi_n,\xi_{n+1},\xi_{n+2})$.
	We denote the image of these maps by $Z^{n+1}_{[n,n+1,n+2]}$ and $Z^{n+2}_{[n,n+1,n+2]}$, respectively. These maps are isomorphisms onto an irreducible component of $Z_{[n,n+1,n+2]}$. To conclude, notice that every element of $Z_{[n,n+1,n+2]}$ is in the image of one of these maps as a point $p$ being contained in the length $n+2$ scheme is either contained in the length $n$ scheme or is one of the residual points.
\end{proof}

\begin{prop}\label{prop: [1,n,n+1,n+2] irred}
$S^{[1,n,n+1,n+2]}$ is irreducible of dimension $2n+4$.
\end{prop}

\begin{proof}
The projection map $Z_{[n,n+1,n+2]} \to S^{[n,n+1,n+2]}$ is the flat family induced by the map $\phi_n: S^{[n,n+1,n+2]} \to S^{[n+1,n+2]}$. Since $Z_{[n,n+1,n+2]}$ admits a finite, flat map to a smooth variety, it is Cohen-Macaulay \cite[Corollary to Theorem 23.3]{Matsumura} and therefore equidimensional (as it is Cohen-Macaulay and connected). 
However, by the previous lemma, it is reducible as it is the union of $Z_{[n,n+1,n+2]}^n$, $Z_{[n,n+1,n+2]}^{n+1}$, and $Z_{[n,n+1,n+2]}^{n+2}$. 
Since we know the complement of $Z_{[n,n+1,n+2]}^n$ inside of $Z_{[n,n+1,n+2]}$ is two irreducible components, both of dimension of $2n+4$, $S^{[1,n,n+1,n+2]}$ is equidimensional of dimension $2n+4$.

We have the forgetful map $\phi_{1}: S^{[1,n,n+1,n+2]} \to S^{[n,n+1,n+2]}$ under which every component of $S^{[1,n,n+1,n+2]}$ dominates $S^{[n,n+1,n+2]}$. 
If $S^{[1,n,n+1,n+2]}$ were reducible, this would guarantee the existence of a point $(\xi_n,\xi_{n+1},\xi_{n+2}) \in S^{[n,n+1,n+2]}$ with preimages $(p,\xi_n,\xi_{n+1}, \xi_{n+2})$ and $(q,\xi_n,\xi_{n+1}, \xi_{n+2})$ that lie in different components of $S^{[1,n,n+1,n+2]}$.
Since the reduced schemes are dense in $S^{[n,n+1,n+2]}$, we can take $\xi_{n+2}$ to be reduced. 
However, we know reduced schemes are curvilinear, and there is a unique curvilinear component so $S^{[1,n,n+1,n+2]}$ is irreducible.
\end{proof}


\begin{cor}\label{cor: [1,n+1,n+2],[1,n,n+2] irred}
$S^{[1,n+1,n+2]}$ and $S^{[1,n,n+2]}$ are irreducible and dimension $2n+4$.
\end{cor}

\begin{proof}
The forgetful maps $\phi_{n+2}: S^{[1,n,n+1,n+2]} \to S^{[1,n,n+1]}$ and $\phi_{n+1}: S^{[1,n,n+1,n+2]} \to S^{[1,n,n+2]}$ are surjective and generically finite so the result is immediate from the proposition.
\end{proof}

\subsection{Reducible nested Hilbert schemes}\label{subsec:Reducible}

\begin{theorem}\label{thm: reducible}
$S^{[1,2,\dots,22,23]}$ is reducible and has a component of dimension at least 48.
\end{theorem}

The crucial step is explicitly constructing a component of the punctual nested Hilbert scheme $(\mathbb{A}^2)_0^{[1,\dots,23]}$ of dimension 46.

\begin{lemma} The punctual nested Hilbert scheme $(\mathbb{A}^2)_0^{[1,2,\dots,22,23]}$ has an irreducible component of dimension at least 46.
\end{lemma}
\begin{proof}
Consider the following ideals:
\begin{align*}
&I_1 = \mathfrak{m},  \\
&I_2 = (x+a_1y)+ \mathfrak{m}^2, \\ 
&I_3 = \mathfrak{m}^2, \\
&I_4 = (xy+a_2y^2,x^2+a_3y^2),\\ 
&I_5 =(x^2+a_4xy+b_1y^2)+ \mathfrak{m}^3,  \\
&I_6 = \mathfrak{m}^3,\\ 
&I_7 = (x^3+a_5y^3,x^2y+a_6y^3,xy^2+a_7y^3)+\mathfrak{m}^4,  \\
&I_8 = (x^3+a_8xy^2+b_2y^3,x^2y+a_9xy^2+b_3y^3)+\mathfrak{m}^4,\\ 
&I_9 = (x^3+a_9x^2y+b_4xy^2+b_5y^3)+\mathfrak{m}^4, \\
&I_{10} = \mathfrak{m}^5,\\
&I_{11} = (x^4+a_{10}y^4,x^3y+a_{11}y^4,x^2y^2+a_{12}y^4,xy^3+a_{13}y^4)+\mathfrak{m}^5, \\
&I_{12} = (x^4+a_{14}xy^3+b_6y^4,x^3y+a_{15}xy^3+b_7y^4,x^2y^2+a_{16}xy^3+b_8y^4)+\mathfrak{m}^5,\\
&I_{13} = (x^4+a_{17}x^2y^2+b_9xy^3+b_{10}y^4,x^3y+a_{18}x^2y^2+b_{11}xy^3+b_{12}y^4)+\mathfrak{m}^5, \\
&I_{14} = (x^4+a_{19}x^3y+b_{13}x^2y^2+b_{14}xy^3+b_{15}y^4,x^3y+a_{20}x^3y+b_{16}x^2y^2+b_{17}xy^3+b_{18}y^4)+\mathfrak{m}^5, \\
&I_{15} = \mathfrak{m}^5,\\ 
&I_{16} = (x^5+a_{21}y^5,x^4y+a_{22}y^5,x^3y^2+a_{23}y^5,x^2y^3+a_{24}y^5,xy^4+a_{25}y^5)+\mathfrak{m}^6,\\
&I_{17} = (x^5+a_{26}xy^4+b_{19}y^5,x^4y+a_{27}xy^4+b_{20}y^5,x^3y^2+a_{28}xy^4+b_{21}y^5,x^2y^3+a_{29}xy^4+b_{22}y^5)+\mathfrak{m}^6, \\
&I_{18} = (x^5+a_{30}x^2y^3+b_{23}xy^4+b_{24}y^5,x^4y+a_{31}x^2y^3+b_{25}xy^4+b_{26}y^5,\\&\;\;\;\;\;\;\;\;\;\;x^3y^2+a_{32}x^2y^3+b_{27}xy^4+b_{28}y^5)+\mathfrak{m}^6,\\
&I_{19} = (x^5+a_{33}x^3y^2+b_{29}x^2y^3+b_{30}xy^4+b_{31}y^5,x^4y+a_{34}x^3y^2+b_{32}x^2y^3+b_{33}xy^4+b_{34}y^5)+\mathfrak{m}^6, \\
&I_{20} = (x^5+a_{35}x^4y+b_{35}x^3y^2+b_{36}x^2y^3+b_{37}xy^4+b_{38}y^5)+\mathfrak{m}^6,\\
&I_{21} = \mathfrak{m}^6, \\
&I_{22} = (x^6+a_{36}y^6,x^5y+a_{37}y^6,x^4y^2+a_{38}y^6,x^3y^3+a_{39}y^6,x^2y^4+a_{40}y^6,xy^5+a_{41}y^6)+\mathfrak{m}^7, \text{ and }\\
&I_{23} = (x^6+a_{42}xy^5+b_{39}y^6,x^5y+a_{43}xy^5+b_{40}y^6,x^4y^2+a_{44}xy^5+b_{41}y^6,x^3y^3+a_{45}xy^5+b_{42}y^6,\\
&\;\;\;\;\;\;\;\;\;\;x^2y^4+a_{46}xy^5+b_{43}y^6)+\mathfrak{m}^7 \\
\end{align*}
where the $a_i$ can be chosen independently and each $b_j$ is uniquely determined by a choice of the $a_k$ and by the relation $I_{i+1} \subset I_i$.
Given that and that each $I_i$ cuts out a scheme of length $i$ supported at the origin, the locus 
\[\{(Z_1,\dots,Z_{23}) : Z_i = \mathbb{V}(I_i)\text{ for some choice of }a_j\}\] is contained in $\mathbb{A}_0^{2[1,\dots,23]}$ and
has dimension 46 so the result follows.
\end{proof}

\begin{proof}[Proof (of Theorem \ref{thm: reducible})]
Since $\mathbb{A}_0^{2[1,2,\dots,22,23]}$ has an irreducible component of dimension at least 46, it is immediate to see that $S^{[1,2,\dots,22,23]}$ has a component of dimension at least 48.
Since the curvilinear component is of dimension 46, the reducibility follows
\end{proof}

\begin{cor}
$S^{[n_{1},\dots,n_k]}$ is reducible for $k>22$ (and has a component of dimension at least $2n_k+2$).
\end{cor}

\begin{proof}
Since, $S^{[1,\dots,22,23]}$ has a component of dimension at least 48, $S^{[n_1,\dots,n_23,\dots, n_k]}$ has a component of dimension at least $48+2(n_1-1)+\sum_{i=2}^23(n_i-n_{i-1}-1)*2+\sum_{i=24}^k 2(n_i-n_{i-1}) = 2n_k+2$ by taking any tuple of subschemes in the at least 48 dimensional component and extending it by adding general points to each subscheme to achieve the desired $n_i$.
Again, the curvilinear component is dimension $2n$, so the reducibility follows.
\end{proof}

This argument can be improved to $k>21$, if you exclude the case of $S^{[1,\dots,22]}$.

\begin{cor}
$S^{[n_{1},\dots,n_k]}$ is reducible for $k>21$ and $n_1>1$ (and has a component of dimension at least $2n_k+2$).
\end{cor}

\begin{proof}
The nested Hilbert scheme $S^{[1,n_{1},\dots,n_k]}$ is reducible and has a componenet of dimension at least $2n_k+2$ by the previous corollary.
Since the map $\phi_1:S^{[1,n_{1},\dots,n_k]} \to S^{[n_{1},\dots,n_k]}$ is surjective and finite, the result follows.
\end{proof}

\section{The RKS Map and Smoothness of $S^{[n,n+1]}$}\label{sec: smoothness of one step}

In this section, we review some machinery from \cite{Hir} which allows one to establish smoothness of certain blowups. 
We will use these results in the next section to construct our resolution of $S^{[n,n+1,n+2]}$. 
We conclude the section by giving a streamlined version of Tikhomirov's proof that $S^{[n,n+1]}$ is smooth. 

Suppose $Z \subseteq X$ is a codimension two Cohen-Macaulay subscheme of a smooth variety $X$. Then $\II_Z$ admits a resolution
\begin{equation}\label{eq:idealRes}
    0 \to E \fto{u} F \to \II_Z \to 0
\end{equation}
where $E$ and $F$ are vector bundles. The map $u$ gives rise to a section of the projection map $D = \PP (\Hom(E,F)^\vee) \to X$. 
By abuse of notation, we refer to this section as $u$. The variety $D$ is stratified by rank varieties
\[
    D_r = \{(x,\phi_x) : \rank(\phi_x) \leq r\}. 
\]
For some $x \in X$, let $y = u(x) = (x,\langle u^\vee|_x \rangle)$. Let $r$ be the rank of $u|_x$. It is well known (see e.g. \cite{Hir}) that $D_r$ is smooth away from $D_{r-1}$ and the normal bundle at $y$ is given by
\[
    N_{D_r/D}|_y \cong \Hom_k(\ker(u^\vee|_x), \coker(u^\vee|_x)).
\]

\begin{deff}\label{def:rankKodairaSpencerMap}
    The \emph{rank Kodaira-Spencer (RKS) map} is the composition
    \[
        \theta_x : T_xX \fto{du} T_yD \to N_{D_r/D}|_y = \Hom_k(\ker(u^\vee|_x), \coker(u^\vee|_x))
    \]
    where the latter map is the natural quotient map. 
\end{deff}

There is a convenient smoothness criterion for $\PP(\II_Z)$ in terms of the RKS map.

\begin{prop}[{\cite[Proposition 3.5]{Hir}}]\label{prop:SmoothnessCriterion}
    Let $(x,\langle v \rangle) \in \PP (\II_Z)$ (note that $v\in \ker(u^\vee|_x)$). Suppose the composition
    \[
        T_xX \fto{\theta_x} \Hom_k(\ker(u^\vee|_x), \coker(u^\vee|_x)) \to \Hom_k(\langle v \rangle, \coker(u^\vee|_x))
    \]
    where the first map is RKS and the second map is the natural restriction, is surjective. Then $\PP (\II_Z)$ is smooth at $(x,\langle v \rangle)$.
\end{prop}

By applying $\Hom_{X}(-,k(x))$ to the exact sequence \eqref{eq:idealRes}, we get an exact sequence
\[
    0 \to \Hom_X(\II_Z, k(x)) \to F|_x^\vee \fto{u^\vee|_x} E_x^\vee \to \Ext^1_X(\II_Z,k(x)) \to 0.
\]
Thus $\Hom_X(\II_Z,k(x)) \cong \ker(u^\vee|_x)$ and $\Ext^1_X(\II_Z,k(x)) \cong \coker(u^\vee|_x)$. Moreover, one sees from the Koszul resolution that $T_xX = \Ext^1_X(k(x),k(x))$. Thus, we can rewrite the RKS map as a map
\[
    \theta_x : \Ext_X^1(k(x),k(x)) \to \Hom_k(\Hom_X(\II_Z,k(x)), \Ext_X^1(\II_Z,k(x)))  \cong \Ext^1(\Hom(\II_Z , k(x)) \otimes \II_Z, k(x)).
\]

\begin{prop}[{\cite[\S2]{Tik}}]\label{prop:RKS=Yoneda}
    The RKS map
    \[
        \theta_x : \Ext_X^1(k(x),k(x)) \to \Ext^1(\Hom(\II_Z , k(x)) \otimes \II_Z, k(x))
    \]
    is obtained by applying $\Ext^1_X(-,k(x))$ to the evaluation map $\Hom(\II_Z,k(x)) \otimes \II_Z \to k(x)$.
\end{prop}

\subsection{Smoothness of $S^{[n,n+1]}$}
Recall that the natural map $\res_{[n,n+1]} \times \pi_n : S^{[n,n+1]} \to S \times S^{[n]}$ is the blowup of the universal family $S^{[1,n]}$. Since $S^{[1,n]}$ is a Cohen-Macaulay subscheme of codimension two in $S \times S^{[n]}$, we can apply the machinery above to establish smoothness of $S^{[n,n+1]}$. 
In particular, we have a resolution
\begin{equation}\label{eq:AppendixRes}
    0 \to E \fto{u} F \to \II_{S^{[1,n]}} \to 0
\end{equation}
where $E,F$ are vector bundles. 

In what follows, we assume that $S$ is a smooth surface. Since smoothness of $S^{[n,n+1]}$ is local on $S$, we may assume that $\pic^0(S)$ is trivial. We let $X = S\times S^{[n]}$ and $Z = S^{[1,n]}$. We denote by $\pi_1 : X \to S$ and $\pi_2 : X \to S^{[n]}$ the projection maps.

\begin{thm}[{\cite{Tik}}]\label{thm:OneStepIsSmooth}
    Let $x = (p,\xi) \in X = S \times S^{[n]}$. The RKS map (see Definition \ref{def:rankKodairaSpencerMap})
    \[
        \theta_x : \Ext^1_X(k(x),k(x)) \to \Ext_X^1\left( \Hom_X(\II_{Z},k(x)) \otimes \II_Z , k(x)\right)
    \]
    is surjective. Thus, $S^{[n,n+1]}$ is smooth.
\end{thm}

In fact, Tikhomirov proves a stronger statement. Namely, the restriction of $\theta_x$ to the subspace $T_{\xi}S^{[n]}$ is surjective. To proceed, we transfer this statement to a cohomological statement on the surface. Since $\pic^0(S)$ is trivial, we have an isomorphism 
\[
    \psi : T_{\xi}S^{[n]} = \Ext^1_{S^{[n]}}(k(\xi),k(\xi)) \fto{\sim} \Ext^1_S(\II_{\xi}, \II_{\xi}).
\]
In terms of extensions, this isomorphism is given by pulling back along $\pi_2$ and tensoring with $\II_{Z}$. Note that $\pi_2^*k(\xi) \otimes \II_Z = \II_{\xi}$. Thus, given an extension $0 \to k(\xi) \to \eta \to k(\xi) \to 0$ with extension class $e \in \Ext^1_{S^{[n]}}(k(\xi),k(\xi))$, we have the extension
\[
    \psi(e) : 0 \to \II_{\xi} \to \pi_2^*\eta \otimes \II_{Z} \to \II_{\xi} \to 0
\]
of sheaves on $S$. Note that $-\otimes \II_Z$ is exact since $\II_{Z}$ has a minimal free resolution of length two.

\begin{prop}[{cf. \cite[\S3 Proposition 2]{Tik}}]\label{prop:AppendixCD}
    Let $x = (p,\xi) \in X$. We have a commutative diagram where the indicated arrows are isomorphisms
    \begin{center}
        \begin{tikzcd}
            \Ext^1_{S^{[n]}}(k(\xi),k(\xi)) \arrow[r, "\theta_x"] \arrow[dd, "\cong"', "\psi"] & \Ext^1_X(\II_Z \otimes \Hom(\II_Z,k(x)), k(x)) \arrow[d, "\res", "\cong"'] \\
             & \Ext^1_S(\II_{\xi} \otimes \Hom_S(\II_\xi, k(p)), k(p)) \\
            \Ext^1_S(\II_\xi,\II_\xi) \arrow[r,"\rho"] & \Ext^1_S(\II_{\xi} , \II_{\xi} \otimes k(p)) \arrow[u, "\phi"', "\cong"] 
        \end{tikzcd}
    \end{center}
\end{prop}
\begin{proof}
    First, we construct each of these maps. The map $\theta_x$ is the usual RKS map, and $\psi$ is the isomorphism constructed above. The map $\res$ is given by restriction to $S \times \{\xi\}$ and is an isomorphism since the restriction of \eqref{eq:AppendixRes} to $S \times \{\xi\}$ is the minimal (locally) free resolution
    \[
        0 \to E|_{S \times \{\xi\}} \fto{u} F_{S \times \{\xi\}} \to \II_{\xi} \to 0.
    \]
    Thus $\Ext^1_S(\II_{\xi},k(p)) \cong \Ext^1_X(\II_Z,k(x))$ and $\Hom_S(\II_\xi, k(p)) \cong \Hom_X(\II_Z,k(x))$. The map $\rho$ is obtained by applying $\Ext^1_S(\II_\xi,-)$ to the restriction map $\II_\xi \to \II_\xi \otimes k(p)$. The isomorphism $\phi$ arises from the perfect pairing
    \begin{center}
        \begin{tikzcd}
            \Hom(\II_\xi,k(p)) \otimes \left( \II_\xi \otimes k(p) \right) \arrow[r]  &  k(p) 
        \end{tikzcd}
    \end{center}
    given by $(f, a \otimes \lambda) \mapsto \lambda f(a)$. Finally given the duality above, the diagram commutes since $\theta_x$ is obtained by applying $\Ext^1_X(k(x),-)$ to the evaluation map $\II_Z \otimes_X \Hom_X(\II_Z,k(x)) \to k(x)$. 
\end{proof}

The final step is to prove that $\rho$ is in fact surjective. Tikhomirov accomplishes this by completing the local rings in question, thereby reducing to the case $S = \PP^2$. We proceed with an easy application of Hilbert-Burch \cite[Theorem 3.2]{E05}. Note that the kernel of the restriction map $\II_{\xi} \to \II_{\xi} \otimes k(p)$ is the ideal $\II_{\xi} \cdot \mathfrak{m}$ where $\mathfrak{m}$ is the maximal ideal at $p$. Note that this ideal may have much larger colength than $\II_\xi$.

\begin{lemma}[{cf. \cite[\S3 Lemma 4]{Tik}}]\label{lem:AppendixSurjectivity}
    The map $\Ext^1_S(\II_\xi,\II_\xi) \to \Ext^1_S(\II_\xi,\II_\xi \otimes k(p))$ obtained by applying $\Ext^1(\II_\xi,-)$ to the restriction map $\II_{\xi} \to \II_{\xi} \otimes k(p)$ is surjective.
\end{lemma}
\begin{proof}
    We have a commutative diagram
    \begin{center}
        \begin{tikzcd}
            \Ext^1_S(\II_\xi,\II_\xi)  \arrow[d, "\rho"] & \Hom_S(\II_\xi, \OO_{\xi}) \arrow[r, "\cong"] \arrow[l, "\cong"'] & \Ext^1(\OO_\xi, \OO_\xi) \arrow[d, "\rho'"] \\
            \Ext^1_S(\II_\xi, \II_{\xi} \otimes k(p))  & \Ext^1_S(\II_\xi, \II_\xi \otimes k(p)) \arrow[r, "\cong"] \arrow[l, "\cong"'] & \Ext^2(\OO_\xi, \II_\xi \otimes k(p)).
        \end{tikzcd}
    \end{center}
    The map $\rho'$ is obtained from the long exact sequence in $\Ext_S^\bullet(\OO_\xi,-)$ associated to the short exact sequence $0 \to \II_{\xi} \otimes k(p) \to \OO_{\xi'} \to \OO_{\xi} \to 0$. Thus, it suffices to show that the natural map $\Ext^2(\OO_\xi, \OO_{\xi'}) \to \Ext^2(\OO_\xi, \OO_\xi)$ is injective. Let $F_\bullet$ be the minimal free resolution of $\OO_\xi$
    \begin{align*}
	    F_\bullet : 0 \to R^{\oplus a} \fto{M} R^{\oplus a+1} \to R \to 0.
    \end{align*}
    By Hilbert-Burch, $\II_{\xi}$ is generated by the $a\times a$ minors of $M$. Now, $\Ext^i(\OO_\xi, \OO_{\xi})$ is the $i$th homology of the complex
\begin{align*}
    \Hom_R(F_\bullet, \OO_{\xi}) : 0 \to \OO_{\xi} \to \OO_{\xi}^{\oplus a+1} \fto{M^T} \OO_{\xi}^{\oplus a} \to 0.
\end{align*}
Thus $\Ext^2(\OO_{\xi}, \OO_{\xi})$ is $\OO_\xi$. Since $M$ is rank $a$ by assumption, the cokernel of $M^T$ is $\OO_{\xi}$ modulo the $a \times a$ minors of $M$, but this is exactly $\OO_{\xi}$. A nearly identical computation shows that $\Ext^2(\OO_{\xi},\OO_{\xi'})$ is $\OO_{\xi'}$ modulo the $a \times a$ minors of $M$. Since $\II_{\xi}$ is generated by the $a \times a$ minors of $M$, we get $\Ext^2(\OO_{\xi},\OO_{\xi'}) \cong \OO_{\xi}$. It is clear from these complexes that the map on $\Ext^2$ carries $1$ to $1$.
\end{proof}

\begin{proof}[Proof of Theorem \ref{thm:OneStepIsSmooth}]
    Let $x = (p,\xi) \in X$. Since $S^{[n,n+1]}$ is the blowup of $X = S^{[n]}$ in $Z = Z_{[n]}$, we may assume $p \in \xi$. 
    By Proposition \ref{prop:SmoothnessCriterion}, it suffices to show that the RKS map is surjective. Proposition \ref{prop:AppendixCD} shows that it suffices to show that $\rho$ is surjective, which is precisely the statement of Lemma \ref{lem:AppendixSurjectivity}. 
\end{proof}

\section{Resolution of singularities of $S^{[n,n+1,n+2]}$}\label{sec:ResolutionAndSings}

\subsection{Resolution of singularities for $S^{[n,n+1,n+2]}$}\label{subsec:Resolution}

In this section, we construct a resolution of singularities of $S^{[n,n+1,n+2]}$. 
Recall that the map $(\res_{[n+1,n+2]},\phi_{n+2}) : S^{[n,n+1,n+2]} \to S \times S^{[n,n+1]}$ is the projectivization of the ideal sheaf of the subscheme $Z_{[n,n+1]} = \{(p,\xi_n,\xi_{n+1}) : p \in \xi_{n+1}\}$. 
Since $S^{[n,n+1,n+2]}$ is irreducible, $\res_{[n+1,n+2]}$ is in fact the blowup of $S \times S^{[n,n+1]}$ at the locus $Z_{[n,n+1]}$.

\begin{prop}\label{prop:Zn,n+1IsReducible}
	The subscheme $Z_{[n,n+1]}$ has two irreducible components, which are isomorphic to $S^{[1,n,n+1]}$ and $S^{[n,n+1]}$.
\end{prop}
\begin{proof}
	There is an obvious map $S^{[1,n,n+1]} \to Z_{[n,n+1]}$ whose image is the collection of triples $(p,\xi_n,\xi_{n+1})$ with $p \in \xi_n$. This map is an isomorphism onto its image. Since both $Z_{[n,n+1]}$ and $S^{[1,n,n+1]}$ have dimension $2n+2$, the image of this map is an irreducible component. There is another natural map $S^{[n,n+1]} \to Z_{[n,n+1]}$ given by $(\xi_n,\xi_{n+1}) \mapsto (\res(\xi_n,\xi_{n+1}),\xi_n,\xi_{n+1})$. Similarly, this map is an isomorphism onto an irreducible component of $Z_{[n,n+1]}$. To conclude, notice that every element of $Z_{[n,n+1]}$ is in the image of one of these maps.
\end{proof}

Since $Z_{[n,n+1]}$ is reducible, it is unsurprising that $S^{[n,n+1,n+2]}$ is singular, and it suggests an approach to finding a resolution. We obtain a smooth model of $S^{[n,n+1,n+2]}$ by blowing up these irreducible components one at a time. First, we blowup the singular component $W_1 \cong S^{[1,n,n+1]}$ and show that the result is smooth, following the strategy in Section \ref{sec: smoothness of one step}. 
To start that, we need to identify the tangent space of $S^{[n,n+1]}$ at a point. This was done for general nested Hilbert schemes in \cite{CheahThesis}, and we recall the relevant special case below.

\begin{prop}\cite{CheahThesis}
    Suppose $\pic^0(S) = 0$. The tangent space to $S^{[n,n+1]}$ at a pair $(\xi_n,\xi_{n+1})$ is isomorphic to the subspace of $\Ext^1_S(\II_{\xi_n}, \II_{\xi_{n}}) \oplus \Ext^1_S(\II_{\xi_{n+1}}, \II_{\xi_{n+1}})$ consisting of extensions $\eta_n$ and $\eta_{n+1}$ such that the following diagram commutes
    \[
        \begin{tikzcd}
            0 \arrow[r] & \II_{\xi_{n+1}} \arrow[r] \arrow[d] & \eta_{n+1} \arrow[d] \arrow[r] & \II_{\xi_{n+1}} \arrow[d] \arrow[r] & 0 \\
            0 \arrow[r] & \II_{\xi_{n}} \arrow[r] & \eta_{n} \arrow[r] & \II_{\xi_{n}}  \arrow[r] & 0
        \end{tikzcd}
    \]
    Where the outer vertical maps are the natural inclusions, and the middle map is uniquely determined.
\end{prop}
\begin{proof}
    The tangent space to $S^{[n,n+1]}$ at $(\xi_n,\xi_{n+1})$ is the subspace of $\Hom_S(\II_{\xi_n},\OO_{\xi_n}) \oplus \Hom_S(\II_{\xi_{n+1}}, \OO_{\xi_{n+1}})$ consisting of pairs $(\phi_n,\phi_{n+1})$ such that 
    \[
        \begin{tikzcd}
            \II_{\xi_{n+1}} \arrow[d] \arrow[r, "\phi_{n+1}"] & \OO_{\xi_{n+1}} \arrow[d] \\
            \II_{\xi_n} \arrow[r, "\phi_n"] & \OO_{\xi_n}
        \end{tikzcd}
    \]
    commutes, where the lefthand vertical map is the inclusion and the righthand vertical map is the projection \cite{CheahThesis}. For any ideal $\II$ in $\OO_S$, there is a natural map $\Hom_{\OO_S}(\II,\OO_S/\II) \to \Ext^1_{\OO_S}(\II,\II)$ given by the connecting map in the long exact sequence associated to the exact sequence $0 \to \II \to \OO_S \to \OO_S / \II \to 0$. This map is an isomorphism when $\II$ cuts out a zero-dimensional scheme on a smooth surface $S$ with $\pic^0(S) = 0$. The description of the tangent space in the proposition follows from observing the image of this map.
\end{proof}

\begin{reptheorem}{thm:Resolution}
	Let $Z_{[n,n+1]} = W_1 \cup W_2$ be the irreducible components of $Z_{[n,n+1]}$ with $W_1 \cong S^{[1,n,n+1]}$ and $W_2 \cong S^{[n,n+1]}$. Let $X_1$ be the blowup of $S \times S^{[n,n+1]}$ at $W_1$, and $X_2$ be the blowup of $X_1$ at the proper transform $\overline{W}_2$ of $W_2$. Then $X_1,X_2$ are smooth, and $X_2$ is a resolution of singularities of $S^{[n,n+1,n+2]}$.
\end{reptheorem}

We summarize the statement of the theorem in the diagram below where each of the arrows is given by the indicated blowup.
	\begin{center}
		\begin{tikzcd}
			\mathrm{Bl}_{Z_{[n,n+1]}}(S \times S^{[n,n+1]}) = S^{[n,n+1,n+2]} \arrow[d] & X_2 = \mathrm{Bl}_{\overline{W}_2}X_1 \arrow[l] \arrow[d] \\
			S \times S^{[n,n+1]} & X_1 = \mathrm{Bl}_{W_1}(S \times S^{[n,n+1]}) \arrow[l]
		\end{tikzcd}
	\end{center}

\begin{proof}
	First, we claim that $W_1 \cap W_2$ is of codimension 1 in $W_2$. Indeed, the general element is of the form $(s, \xi_n,\xi_{n+1})$ where $\xi_n$ is reduced and $\xi_{n+1}$ has a double point at $s$. We have $2n$ dimensions from choosing $\xi_n$ plus one dimension from the choice of tangent vector at $s$. Since $W_2 \cap W_1$ is of codimension 1 in $W_2$, we see $\overline{W_2} \cong W_2 \cong S^{[n,n+1]}$ is smooth. Thus, it suffices to show that the first blowup is smooth. To do this, we mimic Tikhomirov's proof of smoothness of $S^{[n,n+1]}$
	
	Note that $W_1 \cong S^{[1,n,n+1]}$ is Cohen-Macaulay since it admits a finite, flat map $\phi_{1}: S^{[1,n,n+1]} \to S^{[n,n+1]}$ to a smooth variety. We use Proposition \ref{prop:SmoothnessCriterion} to prove the smoothness of $X_1$. In fact, we will prove that for each $x = (s,\xi_n,\xi_{n+1}) \in S \times S^{[n,n+1]}$, the RKS $\theta_x : T_x(S \times S^{[n,n+1]}) \to \Ext^1_X(\II_{W_1} \otimes \Hom_X(\II_{W_1},k(x)),k(x))$ map is surjective. Using the evident modification of the commutative diagram of Proposition \ref{prop:AppendixCD}, it suffices to show that $T_{(\xi_n,\xi_{n+1})}S^{[n,n+1]}$ surjects onto the vector space $V \subseteq \bigoplus_{i=n}^{n+1} \Ext(\II_{\xi_i}, \II_{\xi_i} \otimes k(p))$ of extensions which are compatible in the sense that there is a commutative diagram
	\[
	    \begin{tikzcd}
            0 \arrow[r] & \II_{\xi_{n+1}} \otimes k(s) \arrow[r] \arrow[d] & \eta_{n+1} \arrow[d, "\gamma"] \arrow[r] & \II_{\xi_{n+1}} \arrow[d] \arrow[r] & 0 \\
            0 \arrow[r] & \II_{\xi_{n}} \otimes k(s) \arrow[r] & \eta_{n} \arrow[r] & \II_{\xi_{n}}  \arrow[r] &     0
        \end{tikzcd}.
	\]
	For any such pair of compatible extensions, there is an extension $0 \to \II_{\xi_n} \to E_n \to \II_{\xi_n} \to 0$ which maps to the extension $\eta_n$ under the map induced by the restriction $\II_{\xi_{n}} \to \II_{\xi_n} \otimes k(s)$. Let $\phi_n : E_n \to \eta_n$ be the induced map. Then let $E_{n+1} = \phi_n^{-1}(\gamma(\eta_{n+1}))$. It is clear that $\II_{\xi_{n+1}} \subseteq E_{n+1}$, and a straightforward diagram chase shows that the cokernel is $\II_{\xi_{n+1}}$ as required. That is, the RKS map is surjective. Thus, $X_1$ is smooth by Proposition \ref{prop:SmoothnessCriterion}.
	\end{proof}

\begin{rem}
    We can give a more geometric description of the spaces $X_1,X_2$ from the theorem above. The blowup $\mathrm{Bl}_{W_1}(S \times S^{[n,n+1]})$ is isomorphic to the variety
    \[
        S^{[n,(n+1)^2]} = \{(\xi_n,\xi_{n+1},\xi'_{n+1}) : \xi_n \subseteq \xi_{n+1} \text{ and } \xi_n \subseteq \xi'_{n+1}\} \subseteq S^{[n]} \times S^{[n+1]} \times S^{[n+1]}.
    \]
    An alternate approach to the smoothness of $X_1 = S^{[n,(n+1)^2]}$ is to mimic the approach of \cite{CheahThesis}. That is, one assumes that $S = \PP^2$. It then suffices to establish smoothness at the Borel fixed points, which are subschemes defined by monomial ideals. We settled on our approach to emphasize the paper \cite{Tik} and the connections with \cite{JiangLeung}. 
    
    The vareity $X_2$ surjects onto $S^{[n,n+1,n+2]}$. Thus, to describe $X_2$ as a variety parameterizing certain collections of subschemes of $S$, each point of $X_2$ should determine a length $n+2$ subscheme. Note that the proper transform of $W_2 \cong S^{[n,n+1]}$ in $S^{[n,(n+1)^2]}$ is the locus triples $(\xi_n,\xi_{n+1},\xi'_{n+1})$ such that $\xi_{n+1} = \xi'_{n+1}$. So away from $\overline{W_2}$, the subscheme $\xi_{n+2}$ is easily determined. Indeed if $(\xi_n,\xi_{n+1},\xi_{n+1}') \notin \overline{W_2}$, then $\II_{\xi_{n+1}} \cap \II_{\xi_{n+1}'}$ is an ideal of colength $n+2$. On the other hand, if $\xi = (\xi_n,\xi_{n+1},\xi_{n+1}) \in \overline{W_2}$, the fiber of the blowup $X_2 \to X_1$ is a $\PP^1$. Each point in $\PP^1$ gives a vector $[v]$ in the projectivized tangent space to $S$ at $\res(\xi_n,\xi_{n+1})$, and we obtain the $(n+2)$-scheme by taking the flat limit of a point colliding with $\xi_{n+1}$ along the tangent vector $v$.

\end{rem}

\subsection{Singularities}\label{subsec:Singularities}
We can use the resolution of singularities of $S^{[n,n+1,n+2]}$ to study the singularities of it and other related nested Hilbert schemes.

\begin{prop}\label{prop:TwoStepIsLCI}
    The two-step nested Hilbert scheme $S^{[n,n+1,n+2]}$ is lci.
\end{prop}
\begin{proof}
    First, note that $S^{[n,n+1,n+2]}$ is smooth away from the exceptional locus of the map $\PP(\II_{Z_{[n,n+1]}}) : S^{[n,n+1,n+2]} \to S \times S^{[n,n+1]}$. Thus, it suffices to show that for each $\zeta \in S \times S^{[n,n+1]}$, the scheme $\PP(\II_{Z_{[n,n+1]},\zeta})$ is Gorenstein.
    Recall that $Z_{[n,n+1]}$ is a codimension two Cohen-Macaulay subscheme of the smooth variety $S \times S^{[n,n+1]}$. So for any $\zeta = (p,\xi_n,\xi_{n+1}) \in Z_{[n,n+1]}$, the local ring $\OO_{Z_{[n,n+1]}, \zeta}$ is a Cohen-Macaulay module over the regular ring $R = \OO_{S\times S^{[n,n+1]},\zeta}$. So we have an exact sequence
    \begin{center}
        \begin{tikzcd}
            0 \arrow[r] & \bigoplus^{b_2} \arrow[r, "M"] R & \bigoplus^{b_1} R \arrow[r] & \II_{Z_{[n,n+1]}, \zeta} \arrow[r] & 0
        \end{tikzcd}
    \end{center}
    Note that $b_1 = b_2 + 1$. This sequence presents $\PP(\II_{Z_{[n,n+1]},\zeta})$ as a subscheme of $\PP^{b_1-1}_R$. Pulling back $M$ along the projection map $\pi : \PP^{b_1-1}_R \to \Spec{R}$ and composing with the tautological quotient, we get the following diagram
    \begin{center}
        \begin{tikzcd}
            \bigoplus^{b_2} \OO_{\PP^{b_1-1}_R} \arrow[dr,"\sigma"] \arrow[r, "\pi^*M"] & \bigoplus^{b_1} \OO_{\PP^{b_1-1}_R} \arrow[d] \\
            & \OO_{\PP^{b_1-1}_R}(1).
        \end{tikzcd}
    \end{center}
    Then $\PP(\II_{Z_{[n,n+1]}})$ is scheme-theoretically defined by the vanishing of $\sigma$ so is cut out by $b_2 = b_1-1$ equations. By Theorem \ref{thm:TwoStepIrreducible}, $\PP(\II_{Z_{[n,n+1]}}) = S^{[n,n+1,n+2]}$ is irreducible of dimension $2n+4$ so has codimension $b_1-1$ in $\PP^{b_1-1}_R$. Thus, $S^{[n,n+1,n+2]}$ is lci.
\end{proof}

Similarly, the resolution of $S^{[n,n+1,n+2]}$ immediately implies that it is log terminal.

\begin{cor}\label{cor:TwoStepRational}
    The two-step nested Hilbert scheme $S^{[n,n+1,n+2]}$ is log terminal.
\end{cor}
\begin{proof}
    Recall that $S^{[n,n+1,n+2]}$ and its resolution $\widetilde{S^{[n,n+1,n+2]}}$ are constructed as follows
    \begin{center}
        \begin{tikzcd}
            S^{[n,n+1,n+2]} = \mathrm{Bl}_{Z_{[n,n+1]}} S \times S^{[n,n+1,n+2]} \arrow[d] & \arrow[l] \widetilde{S^{[n,n+1,n+2]}} = \mathrm{Bl}_{\overline{W_2}} X_1 \arrow[d] \\
            S \times S^{[n,n+1]}  & X_1 = \mathrm{Bl}_{W_1}(S \times S^{[n,n+1]}) \arrow[l]
        \end{tikzcd}
    \end{center}
    where the arrows are blowups and the $W_i$'s are the irreducible components of $Z_{[n,n+1]}$. Hence $\widetilde{S^{[n,n+1,n+2]}}$ is a small resolution. Since $S^{[n,n+1,n+2]}$ is Gorenstein, this implies it is log terminal.
\end{proof}

Since the map $\phi_1 : S^{[1,n,n+1,n+2]} \to S^{[n,n+1,n+2]}$ is finite and flat and the target is Cohen-Macaulay, we obtain the following as an immediate corollary.

\begin{cor}
    The nested Hilbert scheme $S^{[1,n,n+1,n+2]}$ is Cohen-Macaulay.
\end{cor}

From our description of $S^{[n,n+1,n+2]}$ as a local complete intersection in a smooth variety, we obtain a Koszul complex that implies that $S^{[1,n+1,n+2]}$ has rational singularities.

\begin{cor}\label{cor:S[1,n,n+1]RationalSings}
    For $n \geq 2$, the nested Hilbert scheme $S^{[1,n,n+1]}$ has rational singularities.
\end{cor}
\begin{proof}
    Since $S^{[n-1,n,n+1]}$ has rational singularities, it suffices to show that the map $\rho : S^{[n-1,n,n+1]} \to S^{[1,n,n+1]}$ given by 
    \[
        \rho(\xi_{n-1},\xi_n,\xi_{n+1}) = \left( \res(\xi_{n-1},\xi_n), \xi_n,\xi_{n+1} \right)
    \]
    satisfies $R^i\rho_*\OO_{S^{[n-1,n,n+1]}} = 0$ for $i > 0$. As in the proof of Proposition \ref{prop:TwoStepIsLCI}, we have an exact sequence
    \[
        0 \to E \fto{u} F \to \II_{Z_{[n-1,n]}} \to 0
    \]
    where $E,F$ are vector bundles. Then $S^{[n-1,n,n+1]} = \PP(\II_{Z_{[n-1,n]}})$ is the zero-subscheme of the map $pi_F^*E \to \OO_{\PP (F)}(1)$. We get a Koszul complex
    \[
        \cdots \to \bigwedge^2 \pi^* E(-1) \to \OO_{\PP(F)} \to \OO_{S^{[n-1,n,n+1]}} \to 0
    \]
    which implies the requisite vanishing.
\end{proof}

We also use Riemann-Hurwitz to get the following.
\begin{cor}\label{cor: S[1,n,n+1]Not Q Gorenstein}  
    $S^{[1,n,n+1]}$ is not $\QQ$-Gorenstein.
\end{cor}

\begin{proof}
    The forgetful map $\phi_1: S^{[1,n,n+1]} \to S^{[n,n+1]}$ is finite and ramified over the locus \[B_{[n]} = \{(p,\xi_n,\xi_{n+1}): p \subset \xi_n \subset \xi_{n+1} \text{ and }\xi_n \text{ is non-reduced}\}.\]
    There are two Weil divisors that dominate $B_{[n]}$, denoted by $E_{[n,n+1]}^1$ and $E_{[n,n+1]}^2$.
    Each of these generically parameterizes triples where $\xi_n$ consists of $n-2$ reduced points and a double point.
    In $E_{[n]}^2$ ($E_{[n]}^1$), the double point is supported (not) at $p$.
    By an entirely analogous argument to the one in the proof of Theorem 7.6 of \cite{Fogarty2} for the case of $S^{[1,n]}$, these are not $\mathbb{Q}$-Cartier, but their sum is. 
    Let $r_i$ be the ramification index of $E_{[n]}^i$ over $B_{[n]}$, and $K(*)$ be the function field of $*$. From the identity \[r_1\left[K\left(E_{[n]}^1\right):K\left(B_{[n]}\right)\right]+r_2\left[K\left(E_{[n]}^2\right):K\left(B_{[n]}\right)\right] = n,\] we see that $r_1=2$ and $r_2=1$ since $\left[K\left(E_{[n]}^1\right):K\left(B_{[n]}\right)\right]=1$ and $\left[K\left(E_{[n]}^2\right):K\left(B_{[n]}\right)\right]=n-2$.
    Then Riemann-Hurwitz formula for finite maps gives that \[K_{S^{[1,n,n+1]}} =\phi_1^*\left(K_{S^{[n,n+1]}}\right)+ E_{[n]}^1\] so $K_{S^{[1,n,+2]}}$ cannot be $\mathbb{Q}$-Cartier.
\end{proof}

\section{Further geometric results}\label{sec: geo results}
In this section, we study a bit of the birational geometry of $S^{[n,n+1,n+2]}$ when $S$ is a smooth surface.
In particular, we compute its Picard group and its canonical divisor when $h^0(S,\mathcal{O}_S) = 0$.
The results may of course be generalized to the case of surfaces with irregularity with a bit of care.

\begin{thm}\label{thm:PicardGroup}
Let $S$ be a smooth surface with $h^1(S,\mathcal{O}_S) = 0$.
Then \[\mathrm{Pic}\left(S^{[n,n+1,n+2]}\right) \isom \mathrm{Pic}(S)^3 \bigoplus \mathbb{Z}^3\]
\end{thm}

\begin{proof}
By \cite{RY}, $\mathrm{Pic}\left(S^{[n,n+1]}\right) = \mathrm{Pic}(S)^2 \bigoplus \mathbb{Z}^2$.

The exceptional divisor of the blow up is irreducible by Theorem \ref{thm:TwoStepIrreducible},
so \[\mathrm{Pic}(\mathrm{Bl}_{W_1}) \isom \mathrm{Pic}(S\times S^{[n,n+1]}) \oplus \mathbb{Z} \isom \mathrm{Pic}(S)^3\oplus \mathbb{Z}^3.\]
Then $\widetilde{S^{[n,n+1,n+2]}}$ is the blow up of $\mathrm{Bl}_{W_1}$ along a smooth irreducible subvariety of codimension 2, so we have \[\mathrm{Pic}\left(\widetilde{S^{[n,n+1,n+2]}}\right) \isom \mathrm{Pic}(\mathrm{Bl}_{W_1})\oplus \mathbb{Z} \isom \mathrm{Pic}(S)^3 \bigoplus \mathbb{Z}^4.\]
Since $\widetilde{S^{[n,n+1,n+2]}}$ is smooth, this is also its Weil divisor group.
The resolution of singularities is an isomorphism in codimension two so that is also the Weil divisor group of $S^{[n,n+1,n+2]}$.

As the singularities of $S^{[n,n+1,n+2]}$ are in codimension two and it is Cohen-Macaulay by Theorem 5.8.6 of \cite{EGAIV}, it is normal.
Since it is normal, the Picard group injects into the Weil divisor group.

Since the exceptional divisor is Cartier, the Picard group of $S^{[n,n+1,n+2]}$ contains $\mathrm{Pic}(S)^3 \bigoplus \mathbb{Z}^3$ where those are generated by the exceptional divisor and the pullbacks from $S\times S^{[n,n+1]}$.
This part of the Picard group is exactly one copy of $\mathbb{Z}$ smaller than the Weil divisor group.
Since the resolution of $S^{[n,n+1,n+2]}$ is a small contraction, it is not $\mathbb{Q}$-factorial.
This implies that the remaining factor of the Weil divisor group must not have any $\mathbb{Q}$-Cartier divisors in it.
As a result, the Picard group is exactly what was claimed.
\end{proof}

We can also give an alternative description of Picard group.
It is easy to see that $\pi_n^*$, $\pi_{n+1}^*$, and $\pi_{n+2}^*$ are injective maps of Picard groups.
Further, it is not difficult to see that the pullbacks of divisors from any pair of distinct Hilbert schemes are numerically distinct.
In other words, the map $(\pi_n^*,\pi_{n+1}^*,\pi_{n+2}^*)$ is injective.
Since the Picard groups are free group of the same size, we get
\[\mathrm{Pic}\left(S^{[n,n+1,n+2]}\right) \isom \mathrm{Pic}\left(S^{[n]}\right)\oplus \mathrm{Pic}\left(S^{[n+1]}\right) \oplus \mathrm{Pic}\left(S^{[n+2]}\right).\] 

\begin{prop}
The class of the canonical divisor of $S^{[n,n+1,n+2]}$ is \[K_{S^{[n,n+1,n+2]}} = \pi_{n+2}^*(K_S)+\pi_{n+2}^*(B[n+2])-\pi_{n}^*(B[n]).\]
\end{prop}

\begin{proof}
The Hilbert scheme is a crepant resolution of the symmetric product \cite{Beauville} so its canonical divisor is $K_S[n] = \mathrm{hc}^*(K_S^{\boxtimes n})$ where $\boxtimes$ means the descent to the symmetric product of the tensor product of the pullback of $K_S$ along each projection of $S^n$ to $S$. 
As $S^{[n,n+1]}$ is the blow-up of $S \times S^{[n]}$ along a codimension two subvariety which is smooth in codimension one (codimension three in the product), its canonical divisor is \[K_{S^{[n,n+1]}} = \mathrm{res}^*(K_S)+\pi_n^*(K_S[n])+E_{(\mathrm{res},\pi_n)} = \pi_{n+1}^*(K_S[n+1])+E_{(\mathrm{res},\pi_n)}.\] 
Since we are blowing up codimension two varieties which are smooth in codimension one (codimension three in the product), we see that $\widetilde{S^{[n,n+1,n+2]}}$ has canonical bundle
\[K_{\widetilde{S^{[n,n+1,n+2]}}} = \mathrm{res}_{[n+1,n+2]}^*(K_S)+\phi_{[n+2]}^*(K_{S^{[n,n+1]}})+E_1+E_2\] where $E_1$ and $E_2$ are the exceptional divisors of the successive blow-ups in the construction of $\widetilde{S^{[n,n+1,n+2]}}$.
Since the resolution is a small resolution, it is crepant so the canonical is unchanged in the Picard group of $S^{[n,n+1,n+2]}$.
On there, this can be simplified to 
\[K_{S^{[n,n+1,n+2]}} = \pi_{n+2}^*(K_S)+\pi_{n+2}^*(B[n+2])-\pi_{n}^*(B[n])\] since $\pi_{n+2}^*(B[n+2])$ is the union of $E_1$, $E_2$, the pullback of the exceptional divisor of the blow-up defining $S^{[n,n+1]}$, and $\pi_{n}^*(B[n])$.
\end{proof}



\bibliographystyle{alphanum}
\bibliography{NestedHilbertSchemesBib}
\end{document}